\documentclass{article}
\usepackage{amsfonts,amsmath,amstext,amsbsy,euscript,amssymb,amsthm}
\usepackage{pstricks}
\usepackage{tikz}
 \font\amsy=msbm10
\def\IZ{\hbox{\amsy\char'132}}
\def\IR{\hbox{\amsy\char'122}}
\newtheorem{theorem}{Theorem}
\newtheorem{lemma}{Lemma}
\newtheorem{prop}{Proposition}
\newtheorem{remark}{Remark}
\newtheorem{note}{Observation}
\newtheorem{example}{Example}
\usepackage[all]{xy}

\newtheorem{definition}{Definition}
\usepackage{tikz}
\newcommand*\circled[1]{\tikz[baseline=(char.base)]{
            \node[shape=circle,draw,inner sep=1pt] (char) {#1};}}

\newcommand{\tokg}{\begin{pspicture}(12pt,9pt)\psframe[linewidth=1pt,framearc=.8,fillstyle=solid, fillcolor=myGreen](1,0.3)\end{pspicture}}
\newcommand{\tokr}{\begin{pspicture}(12pt,9pt)\psframe[linewidth=1pt,framearc=.7,fillstyle=solid, fillcolor=red](1,0.3)\end{pspicture}}

\newcommand{\es}{\begin{pspicture}(12pt,3pt)\psframe[linewidth=1 pt,framearc=.7,fillstyle=solid, fillcolor=brown](3.95,0.2)\end{pspicture}}

\newrgbcolor{purple}{0.7 0.2 0.7}
\newrgbcolor{orange}{1.0 0.5 0.0}
\newrgbcolor{myGreen}{0 1 0.0}

\begin{document}
\title{Chromatic Nim\\ { finds a game for your solution}}
\author{Michael J. Fisher\footnote{Dept. of Mathematics, West Chester University, email: mfisher@wcupa.edu}, Urban Larsson\footnote{Dept. of Mathematics and Statistics, Dalhousie University, supported by the Killam trust, email: urban031@gmail.com}}
\maketitle

\begin{abstract}
We play a variation of Nim on stacks of tokens. Take your favorite increasing sequence of positive integers
and color the tokens according to the following rule. 
Each token on a level that corresponds to a number in the sequence is colored red; if the level does not correspond to a number in the sequence, color it green.
Now play Nim on a arbitrary number of stacks with the extra rule: if all top tokens are green, then you can make \emph{any} move you like. 
On two stacks, we give explicit characterizations for winning the normal play version for some popular sequences, such as Beatty sequences and the evil numbers corresponding to the 0s in the famous Thue-Morse sequence. We also propose a more general solution which depends only on which of the colors `dominates' the sequence. Our construction resolves a problem posed by Fraenkel at the BIRS 2011 workshop in combinatorial games.
\end{abstract}

\section{Introduction}
At the workshop in combinatorial games in BIRS 2011, Aviezri Fraenkel posed the following intriguing problem: find nice (short/simple) rules for a 2-player combinatorial game for which the $\mathcal{P}$-positions are obtained from a pair of complementary Beatty sequences \cite{Be}. We begin by solving this problem, by defining a class of heap games, dubbed  Bi-Chromatic Nim, or just Chromatic Nim, and then later in Section~\ref{S:4}, we explain some background to the problem. In Section~\ref{S:3}, we solve a similar game on arithmetic progressions. In Section~\ref{S:5}, we discuss the general environment for  Chromatic Nim on two heaps. At last, in Section~\ref{S:6}, we study the famous evil numbers, also known as the indexes of the 0s in the Thue-Morse sequence. 

\section{Bi-chromatic Nim finds a game for your Complementary Beatty solution}\label{S:2}
Let $S$ denote a subset of the positive integers. We let the $i$th token in a stack be \emph{red} if $i\in S$, and otherwise the $i$th token is \emph{green}. We play a take-away game on $k\geqslant 0$ copies of such stacks of various finite sizes. Classical Nim rules are always allowed; any number of tokens can be removed from precisely one of the stacks. In addition, if no heap size belongs to $S$, then \emph{the position is green} and any move is legal; in particular it is now allowed to lower all stacks to 0. Another way to identify  a green position is to look at the stacks from above. If you see only green tokens, then the position is green. Two players alternate moving and a player who cannot move loses. Note that if $S$ is the set of positive integers, then the game is $k$-pile Nim (because all tokens are red). If $S$ is the empty set, then the game is 1-pile Nim, independently of $k$, because all tokens are green (for a reader who likes to compute so-called Grundy values of impartial games, in this special case it obviously means that the Grundy value is the total number of tokens). We call this game $S$-Chromatic Nim.

Let $\beta > 2$ be irrational and let $S=\{\lfloor \beta n\rfloor \}$, for $n$ running over the positive integers. Then the red tokens are determined by: the $i$th token is red if and only if there is an $n$ such that $i=\lfloor \beta n\rfloor $. We play on two stacks, and, since the sequences are regular (as opposed to random), determined by a number $\beta$, we call the game (2-stack) $\beta $-Chromatic Nim.

\begin{figure}[h]
\psset{xunit=1cm}
\psset{yunit=0.8cm}
\begin{center}
\begin{pspicture}(3.2, 2)
\multirput(-.47,-.3 )(2, 0){1}{\es}
\multirput(0,0)(0,0.3){4}{\tokg}
\multirput(2,0)(0,0.3){2}{\tokg}
\multiput(-0.2,0.08)(0,0.3){1}{\tokr}
\multiput(1.8,0.08)(0,0.3){1}{\tokr}
\end{pspicture}
 \caption{The 2-stack $\phi^2$-Chromatic Nim position $(4, 2)$, where $\phi = \frac{1+\sqrt{5}}{2}$.}\label{F:1}
\end{center}

\end{figure}
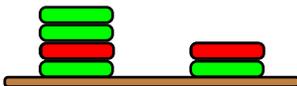
For an example, view Figure~\ref{F:1}. Since the position is not green, then only Nim type moves are possible. The unique winning move is to remove three tokens from the Left most heap. The easiest way to identify a winning move is to make sure precisely one heap is green and the other one is red. Now you also need to count the number of tokens in the respective stack, colored in the same color as the top token. If these two numbers are identical then you have found your winning move. This idea generalizes as we show in Section~\ref{S:5}

\begin{theorem}\label{two}
Let $\beta>2$ be irrational. Then a position $(x,y)$ of 2-stack $\beta$-Chromatic Nim is a previous player winning position if and only if $(x,y)=(\lfloor \alpha n \rfloor, \lfloor\beta n\rfloor)$ or $(\lfloor \beta n \rfloor, \lfloor\alpha n\rfloor)$, for some $n\in \mathbb{N}$, and where 
\begin{align}\label{Beatty}
\frac{1}{\alpha} + \frac{1}{\beta} = 1.
\end{align}
\end{theorem}

\begin{proof}
Recall Beatty's theorem \cite{Be}: if (\ref{Beatty}) is satisfied, then the sequences $(\lfloor \alpha n\rfloor )$ and $(\lfloor \beta n\rfloor )$ are complementary, for $n>0$, each positive integer occurs in precisely one of the sequences and only once in this sequence. Suppose first that $x=\lfloor \alpha m \rfloor$ and $y = \lfloor\beta m\rfloor$, for some $m\in \mathbb{N}$. If $m=0$, we are done, so suppose that $m>0$. Observe that no position of this form is green, since one of the coordinates is $\lfloor \beta m\rfloor$, a red stack height. Hence it suffices to show that there is no Nim option of the same form. Note that (\ref{Beatty}) together with $\beta >2$ implies that $\alpha >1$. Hence decreasing just one of the heaps cannot give a position of the same form; it follows since, by Beatty's theorem, the sequences $(\lfloor \alpha n\rfloor )$ and $(\lfloor \beta n\rfloor )$ are complementary. 

Suppose next that the pair $(x, y)$ is not of the given form. If the smaller heap is empty then, the current player removes all tokens in the higher stack as well, which solves this case. Otherwise there are positive integers $m\geqslant n$ such that either\\

\noindent{\sc Case 1:} $x=\lfloor \alpha m \rfloor, y= \lfloor\beta n\rfloor$, $m>n$\\

\noindent{\sc Case 2:} $x=\lfloor \alpha m \rfloor, y=\lfloor\alpha n\rfloor$, $n>0$\\

\noindent{\sc Case 3:} $x=\lfloor \alpha n \rfloor, y=\lfloor\beta m\rfloor$, $m>n$\\

\noindent{\sc Case 4:} $x=\lfloor \beta m \rfloor, y=\lfloor\beta n\rfloor$, $n>0$.\\

Notice that none of the positions represents a position of the form in the theorem, the first and third since the sequences are strictly increasing and the second and fourth by complementarity. Hence, our task is to find a legal move to a position of the form in the theorem, for each case. 

The position $(x,y)$ given by the second case is green and so it is an $\mathcal{N}$-position. This follows from complementarity of the sequences $(\lfloor\alpha i\rfloor)$ and $(\lfloor\beta i\rfloor)$, namely since $x = \lfloor \alpha m \rfloor$ there is no integer $i$ such that $\lfloor \beta i\rfloor = x$ and similarly for $y$. For Case 1, it is clear that the current player can lower the $x$ stack to $x=\lfloor \alpha n \rfloor$.

For the third case, by $m>n$ since $\beta>2$ we get $\lfloor \beta m\rfloor>\lfloor \beta n\rfloor$, so that the desired Nim move on the $y$-stack is to lower it to the position $(\lfloor \alpha n \rfloor, \lfloor\beta n\rfloor)$. The fourth case is similar, but the lowering is on the $x$-stack, motivated by $\lfloor \beta m \rfloor\ge \lfloor\beta n\rfloor >\lfloor \alpha n\rfloor $, which follows since (\ref{Beatty}) gives $1<\alpha<2<\beta $ and by $n>0$. (The latter inequality excludes the terminal position $(x,y) = (0,0)$ which of course is also of the form $(x,y) = (\lfloor \alpha n \rfloor, \lfloor\beta n\rfloor)$, for some $n\in \mathbb{N}$).
\end{proof}


\section{Games with arithmetic progression solutions}\label{S:3}
Next, let us study a generalization of the game rules of $\beta$-Chromatic Nim to $\beta$ an integer, that is, let $\beta\geqslant 2$ be an integer and let $S = \{\beta n\mid n\in \IZ^+ \}$. We have the following perhaps not so surprising result in view of Theorem~\ref{two}. The solution will still consist of complementary sequences, and indeed we have now shifted focus around to the more standard one in CGT, finding a solution for your game. 

\begin{theorem}\label{one}
Let $\beta\geqslant 2$ be an integer. Then a position $(x,y)$ of 2-pile $\beta$-Chromatic Nim, with $x\leqslant y$, is a previous player winning position if and only if $(x,y)=(0,0)$ or 
\begin{align}\label{xy}
(x, y) &= ( \beta n+t, (\beta-1)(\beta n +t) + t)\\
&= (\beta n + t, \beta ((\beta -1)n+t)), \label{xy2}
\end{align}
for some $n\in \mathbb{N}$ and some $t\in\{1,\ldots ,\beta -1\}$. 
\end{theorem}

\begin{proof} 
Note first that, by definition of $t$, $(\beta -1)n+t$ takes on all the positive integers. Therefore the $y$-coordinates will take on precisely all multiples of $\beta$. For the same reason, the $x$-coordinates will take on precisely the complement of this set. Let $\mathcal{P}'$ denote all positions $(x,y)$ and $(y,x)$ where $x$ and $y$ are defined by  (\ref{xy}). We begin by showing that no option of $(x,y)\in \mathcal{P}'$ is in $\mathcal{P}'$. Since the sets of all $x$'s and $y$'s are complementary, it suffices to prove that only Nim type moves are possible. But, by definition, with notation as in (\ref{xy2}), each $x$ is green and each $y$ is red. 

Let us next prove that from each position $(x, y)\not\in \mathcal{P}'$, there is a move to a position in $\mathcal{P}'$. 

For this case, there are positive integers $m\geqslant n$ such that either\\

\noindent{\sc Case 1:} $x=\beta m+t, y= \beta ((\beta -1)n+t)$, $m>n$\\

\noindent{\sc Case 2:} $x=\beta m+t, y=\beta n+t$, $n>0$\\

\noindent{\sc Case 3:} $x=\beta n +t, y=\beta ((\beta -1)m+t)$, $m>n$\\

\noindent{\sc Case 4:} $x=\beta ((\beta -1)m+t), y=\beta ((\beta -1)n+t)$, $n>0$.\\

For case 1, we can reduce the $x$-stack to $\beta n+t$. For case 2, both stacks are green, so a move to $(0,0)$ is possible. For case 3, we can reduce $y$ to $\beta ((\beta -1)n+t)$. Finally, for case 4, since $\beta\geqslant 2$, $\beta ((\beta -1)m+t)\geqslant \beta ((\beta -1)n+t)\geqslant \beta (n+t)>\beta n+t$, so the desired move to $(\beta n + t, \beta ((\beta -1)n+t))$ is possible.
\end{proof}

\section{Discussion of the origin to the problematic solution}\label{S:4}

Typically, game rules of combinatorial games are short and easy to learn, but not always. One should be able to learn the rules of a game without a degree in mathematics. The distinction we are speaking of is Play-games versus Math-games. We contribute Play-game rules to an original Math-game problem. The new element is the coloring of the tokens. This takes care of uncountably many problems (disguised as one problem). We already know that there is a countably infinite family of Play-game rules for $\mathcal{P}$-positions of distinct complementary Beatty sequences \cite{F82} and that family has been expanded via continued fractions in \cite{DR, LW}. But Wythoff Nim \cite{W} is the origin of these type of questions. Nim on two heaps provides a nearly trivial mimicking winning strategy and the $\mathcal{P}$-positions are all positions with equal heap sizes. By adjoining these Nim $\mathcal{P}$-positions as moves in a new game, Wythoff discovered, that the new $\mathcal{P}$-positions will be described by half lines of slope the Golden ratio and its inverse.  In fact, the game in Figure~\ref{F:1} is $\mathcal{P}$-equivalent to Wythoff Nim. Now, the challenge of finding game rules for \emph{any} complementary pair of Beatty sequences was posed in \cite{DR} and resolved in \cite{LHF}. There was a proviso to the solution in \cite{DR}; the game rules must be \emph{invariant}, and this is a new notion to an old description of vector subtraction games from \cite{G}. Many game rules are non-invariant (perhaps the most famous of them all is Fibonacci Nim), but the classical ones (Subtraction games, Nim, Wythoff Nim) are invariant in the sense that a rule does not depend on which position it was  moved from (apart from the empty-heap condition). The $\star$-operator defined in \cite{LHF} produces invariant games, but only with exponential complexity in log of heap sizes. So the proviso game rules being invariant is nice in one way, but on the other hand, the obtained games cannot easily  be played by human beings. The exponential complexity is decreased to polynomial ditto in \cite{FL}, but where invariance was relaxed to 2-invariance, a special restricted family of variant games; but although the game rules are polynomial in succinct input size, they remain Math games. Here we study simple game rules: true Play-games. They have a somewhat surprising solution (although not as surprising as the solution Wythoff originally discovered in his variation of the game of Nim). 

\section{Other Chromatic Nim games}\label{S:5}

Since Chromatic Nim was capable enough to solve the problem posed at the BIRS 2011 workshop in combinatorial games, we got interested in what properties the game might have in a somewhat more general setting. Let us discuss a natural generalization of the games and sequences from Section~\ref{S:2}, still just on 2 stacks. Given a set $S$, let us call our game $S$-Chromatic Nim. The following lemma will allow us to resolve any 2-stack game for sequences with a surplus of green tokens, without further knowledge of the sequence. Let us regard the set $S$ as an increasing sequence of integers $S=\{s_i\}_{i>0}$ and let $\overline S=\{\overline s_i\}_{i>0}$ denote the unique increasing sequence \emph{complementary} to $S$ on the positive integers (that is if $n$ is a positive integer, then $n$ is in precisely one of the sequences). Then $S$ is \emph{green-dominated} if, for all $i$, $\bar s_i < s_i$, and $S$ is \emph{red-dominated} if, for all $i$, $\bar s_i > s_i$. See also paper \cite{L2014} for a similar construction. Hence it is clear that a sequence cannot be both red-dominated and green-dominated, and of course, `most' increasing sequences are neither. 
Note that for example $\beta$-Chromatic Nim from Section~\ref{S:2} is green-dominated, since $\beta>2$. We have the following lemma for any green-dominated game $S$-Chromatic Nim.

\begin{lemma}\label{lem:1}
Suppose that $S$-Chromatic Nim is green-dominated. Let $$\{(a_i,b_i), (b_i,a_i) \mid i\in \mathbb{N} \}$$ denote its set of $\mathcal{P}$-positions, where for all $i\geqslant 0$, $a_i\leqslant b_i$. Then, for all $i>0$, the $i$th green token from below is the $a_i$th token and the $i$th red token from below is the $b_i$th token. Therefore, for all $i$, $a_i<b_i$, which implies that all monochromatic positions are $\mathcal{N}$-positions. 
\end{lemma}

\begin{proof}
First of all it is clear that if both heaps are green, then there is a move to $(0,0)$, so we assume first that one of the heaps is red and the other green. If in addition, the red heap contains the same number of red tokens as the number of green tokens in the green heap, then there is no nim-type move to a position of the same form. But these are the only legal type of moves, so the ``$\mathcal{P}$ cannot go to $\mathcal{P}$ property" is satisfied. Now, if the number of red tokens in the red heap is different from the number of green tokens in the green heap, then we have to find a candidate $\mathcal{P}$-position to move to. If there are more red tokens in the red heap than there are green tokens in the green heap, then there is a nim-type move that equalizes the numbers. Hence, only the case for two red heaps remains to be considered. One of the heaps contains no more red tokens than the other. Then, because of the green-dominated property, the other heap contains at least as many green tokens as the number of red ones in the former. Hence a Nim-type move suffices to reduce the taller red heap to a green heap with as many green tokens as the number of red ones in the previously smaller red heap. The base case is that since the game is green-dominated, when the first red token appears, then there is a green below, so the above proof applies.
\end{proof}

How do you play to win a green-dominated game? Let us summarize the proof of Lemma~\ref{lem:1}.
\begin{prop}\label{prop:easyN}
If the heaps have different colors and the red heap has $r$ red tokens and the green heap has $g$ green tokens, with $r=g$, then there is no winning move for the current player. Otherwise the first player should remove $r-g$ red tokens from the red heap if $r>g$ and $g-r$ green tokens from the green heap if $g>r$, in either case keeping the color of the changed heap the same. If both heaps are green you move to $(0,0)$. If both heaps are red and $r$ is the number of red tokens in the smaller heap, then play in the larger (or equally sized) heap so that it becomes green with $r$ green tokens. 
\end{prop}

\begin{proof}
This is a direct consequence of Lemma \ref{lem:1}. Notice again, for the last sentence, this is always possible, because of the green-dominated property.
\end{proof}

For $\beta$-Chromatic Nim with a rational $\beta >2$, we know a winning strategy via Proposition~\ref{prop:easyN}, but do not yet have a complete characterization extending Theorems~\ref{two} and \ref{one}. 


When $S$ is red-dominated, then Lemma~\ref{lem:1} and Proposition~\ref{prop:easyN} are no longer true, and it is easy to see because the base case fails. Let $\beta=3/2$. Then, for $\beta$-Chromatic Nim, $S=\{1,3,4,6,7,9,\ldots\}$. This sequence is not green-dominated since, for example, the stack with one token has one red token and no green one. It is red-dominated because $\bar S=\{2,5,8,11\ldots\}$, and so $s_i<\bar s_i$ for all $i$. In fact, the smallest non-terminal $\mathcal{P}$-position is the red position $(1,1)$, so the conclusions of Lemma~\ref{lem:1} and Proposition~\ref{prop:easyN} are false. Let us list the first few $\mathcal{P}$-positions of this game, to obtain intuition for the next result on red-dominated sequences. In Figure~\ref{F:2}, we see the three first P-positions of $\frac{3}{2}$-Chromatic Nim, and the pictures illustrate how colors can be either both red or mixed. But there is a simple explanation to this behavior. We use the following definition.

\begin{definition}\label{def:1}
The game $S$-Chromatic Nim is \emph{locally green-dominated} (lgd) at level $d\in S$, if there is some positive integer $k\in \bar S$, such that the $k$-shifted heaps (with the lower $k-1$ tokens removed) induces a local green-dominated game on $d$ tokens, colored according to (a re-indexed) $S\setminus\{0,\ldots , k-1\}$. The local green-domination is maximal if the game is not lgd at level $d+1$.
\end{definition}

In a sense the red-dominated games behave like Nim, on the red tokens, but any local green-dominance has to be compensated for; we can use a local variation of Lemma~\ref{lem:1} to compute the $\mathcal{P}$-positions, until the position is not lgd any longer, at which point the old Nim-strategy will reappear (see the third picture).

\begin{figure}[ht!]
\psset{xunit=.8cm}
\psset{yunit=0.65cm}
\begin{pspicture}(3.2, 2)
\rput(2,0){
\multirput(-.47,-.3 )(2, 0){1}{\es}
\multirput(0,0)(0,0.3){1}{\tokr}
\multirput(2,0)(0,0.3){1}{\tokr}
}
\rput(6.5,0){
\multirput(-.47,-.3 )(2, 0){1}{\es}
\multirput(0,0)(0,0.3){3}{\tokr}
\multirput(2,0)(0,0.3){1}{\tokr}
\multirput(0,.3)(0,0.3){1}{\tokg}
\multirput(2,.3)(0,0.3){1}{\tokg}
}
\rput(11,0){
\multirput(-.47,-.3 )(2, 0){1}{\es}
\multirput(0,0)(0,0.3){4}{\tokr}
\multirput(2,0)(0,0.3){4}{\tokr}
\multirput(0,.3)(0,0.3){1}{\tokg}
\multirput(2,.3)(0,0.3){1}{\tokg}
}

\end{pspicture}
\caption{The first three $\mathcal{P}$-positions of $\frac{3}{2}$-Chromatic Nim.}\label{F:2}

\end{figure}
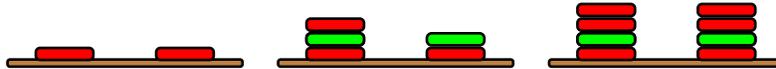

\begin{prop}\label{prop:red}
Consider a red-dominated game $S$ and $d$ a positive integer. Then $(d,d)$ is a $\mathcal{P}$-position if and only if the game is not locally green-dominated at level $d$. Otherwise, we consider the locally green-dominated game beginning at some minimal level $k$, that is, we apply Lemma~\ref{lem:1} to compute the $\mathcal{P}$-positions with level $k$ exchanged for level $1$ (a green token). This computation stops at a level where the lgd is maximal.
\end{prop}
\begin{remark}
In Figure 2, we obtain the first non-zero $\mathcal{P}$-position as $(1,1)$. (Only red tokens behave as Nim.) The second $\mathcal{P}$-position is of the second type (with $k=2$), because colors of heaps are mixed. The way to compute the $\mathcal{P}$-positions in a lgd game is to apply the algorithm for green-dominated games, but here starting with the 2nd green layer of tokens rather than the first red layer. Now, already at level~3 ($d=3$) the lgd becomes maximal, so in fact the green-dominated algorithm terminates in just one step. (Notice here that Definition~\ref{def:1} is satisfied with $k=2$ and $d=3$.)
\end{remark}
\begin{proof}
By Definition~\ref{def:1}, the stack-sizes can be partitioned into two classes. Class~1 is the set of all non-lgd's and Class~2 is the set of all (maximal) lgd's. The latter contains the discrete intervals of the form $\{k,\ldots , d\}$ as in Definition~\ref{def:1}. Since we are discussing red-dominated games the base case is as the left-most picture in Figure \ref{F:2}, level 1 is red. Now, there is a smallest green token, at say level $k$. It will be the first Class 2 token. Since the first red token above this green one exists (by red-dominating), say at level $d>k$, the first Class 2 $\mathcal{P}$-position will be $(k,d)$. The Class 2 $\mathcal{P}$-positions will have to continue as outlined in Lemma~\ref{lem:1} until perhaps the level-$k$ adjusted  green-dominating property fails. (It does not have to fail, because the set $S$ can still be red-dominated, because of the initial layer(s) of red tokens could compensate for example a local green dominated periodic behavior.) If it fails, then there will be a maximal (red) token, say at level $d>k$ for which lgd holds. Then $(d+1, d+1)$ is a $\mathcal{P}$-position. It follows from the fact that the red token at level $d+1$ cannot be paired up with a green token below, because, by Lemma~\ref{lem:1}, they have already been matched up with lower red tokens. Now the proof follows by induction, since the continuation from level $d+1$ onwards is the same as restarting from level 1, with $d+1$ exchanged for 1. (Because only Nim-type moves are allowed, by induction, no lower $\mathcal{P}$-positions can be reached in a single move, and the special move rule for green positions can just as well be used to move to $(d+1, d+1)$, which is $\mathcal{P}$ by induction.
\end{proof}
Now it is easy to combine Proposition~\ref{prop:easyN} with the proof of Proposition~\ref{prop:red} to find your winning move. The clue is to identify the non-lgd sequences to know the excess of red tokens that need to be subtracted in order to use the green-dominated result correctly. But, this is easy by the recursive argument. We note that this general winning strategy requires a bottom-up approach and is therefore slow compared to the results in Sections~\ref{S:2} and \ref{S:3}.

\section{More fractal rules and strategies}\label{S:6}
We say that a non-negative integer is \emph{evil} if it has an even number of ones in its binary expansion;  otherwise it is called \emph{odious}.   In this section, we let $S$ be the set of evil integers and call the resulting game \emph{evil-Chromatic Nim}.  Thus, in this game, the evil integers are red and the odious integers are green.

The following observation is not strictly needed in the proofs to come, but we prove it anyway.
  
\begin{note}\label{evil:4} {There cannot be more than two consecutive odious numbers (evil numbers) in any sequence of consecutive integers.}
\end{note}

\begin{proof} Suppose $n, {n+1}, {n+2}$ are three consecutive odious numbers.  Note that $n$ has to be odd or the parity of the number of 1's for ${n+1}$ will be incorrect.  Because of the constraints on $n$, $n$ has to look like   $$\begin{array}{rrrrr}
(i)  &{n = \underbrace{11\cdots1}_{2i+1 \ 1\text{'s}}}  &\text{or} &{(ii)}  &{n = \underbrace{1\ast \cdots\ast}_{2i \ 1\text{'s}} \circled{0}\underbrace{11 \cdots 1}_{2j+1 \  1\text{'s}}} \\
\end{array}$$
where $(i)$ consists of an odd number of consecutive 1's and $(ii)$ begins from the left with a 1 followed by 1's and 0's, with an even number of 1's to the left of the circled 0 (a $*$ indicates a 0 or a 1).  Then  $$\begin{array}{rrrrr}
(i)  &{ {n+1} = \mathbf{1}\underbrace{00\cdots0}_{2i+1 \ 0\text{'s}}}  &\text{or}  &{(ii)}  &{{n+1} = \underbrace{1\ast \cdots\ast}_{2i \ 1\text{'s}}  \circled{1}\underbrace{00 \cdots 0}_{2j +1 \  0\text{'s}}.}
\end{array}$$  But then, in either case, ${n+2}$ will contain an even number of 1's and thus be evil.  The case for evil numbers is similar. \end{proof}

\begin{definition} Let $k$ be a nonnegative integer and let $U \subseteq [k] \cup \{0\} = \{0, 1, \dots, k\}$ be a subset consisting of consecutive integers.  Then the \textbf{pseudo-chromatic number} of $U$, denoted $\tau(U)$, is defined to be the difference $$(\# \text{ of green numbers in } U)-(\# \text{ of red numbers in } U).$$  In the special case where $U = [k] \cup \{0\}$, we write $\tau(k)$ for $\tau(U)$.
\end{definition}


\begin{lemma}\label{evil:1} For any nonnegative integer $n$, $-1 \leq \tau(n) \leq 1$.  


\end{lemma}

\begin{proof}  The proof will proceed by induction on $n$.  For our base cases, we consider the integers 0 through 7.

$$\begin{array}{c|c|c|r}
\textbf{Integer} &\textbf{Binary Representation} &\textbf{Green/Red} &\tau \\ \hline
0 &000 &\text{red} &-1\\

1 &001 &\text{green} &0 \\

2 &010 &\text{green} &1 \\
 
3 &011 &\text{red} &0 \\

4 &100 &\text{green} &1 \\

5 &101 &\text{red} &0 \\

6 &110 &\text{red} &-1 \\

7 &111 &\text{green} &0 \\

\end{array}$$

Longer lists (of length $2^j$) can be built recursively as follows.

\begin{displaymath}
   \xymatrix{A &{\begin{array}{c|r} {} &\tau \\ \hline \mathbf{00}00 & {-1} \\  \mathbf{00}01 &{0} \\ \mathbf{00}10 & {1} \\   \mathbf{00}11 & {0}  \end{array}}  \ar@/^4pc/[ddd] \\
    B &{\begin{array}{c|r} \mathbf{01}00 &1\\ \mathbf{01}01&0 \\ \mathbf{01}10 &-1 \\ \mathbf{01}11 &0\end{array}}  \ar@(dl,dr)[d] \\ 
   C& {\begin{array}{c|r} \mathbf{10}00 &1\\ \mathbf{10}01&0 \\ \mathbf{10}10 &-1 \\ \mathbf{10}11 &0\end{array}} \\
   D& {\begin{array}{c|r} \mathbf{11}00 & {-1} \\  \mathbf{11}01 &{0} \\ \mathbf{11}10 & {1} \\   \mathbf{11}11 & {0}  \end{array}} }
\end{displaymath}

Notice that the above figure is constructed by prefixing the block of the binary numbers $00, 01, 10,$ and $11$ with $00$, $01$, $10$, or $11$, respectively. Further, observe that blocks $A$, $B$, $C$, and $D$ each have pseudo-chromatic number equal to $0$. Hence, $\tau(15) = 0$. 

Moreover, notice the maps illustrated in the figure from $A$ to $D$ and from $B$ to $C$ preserve the evil/odious quality of each integer and its respective pseudo-chromatic number.

Next assume that our result holds for all $j$ such that $1 \leq j < n$.  Find $m >0$ so that $2^{m-1} \leq n \leq 2^m-1$.  Using the recursive construction described above, we construct the list below of length $2^m$.

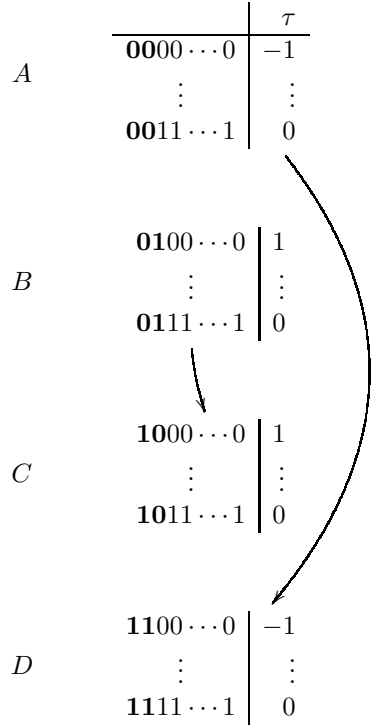
\begin{displaymath}
   \xymatrix{A &{\begin{array}{c|r} {} &\tau \\ \hline \mathbf{00}00 \cdots 0 &-1 \\ \vdots &{\vdots} \\ \mathbf{00}11\cdots 1 &0 \end{array}}  \ar@/^5pc/[ddd] \\
    B &{\begin{array}{c|r} \mathbf{01}00  \cdots 0 &1\\ \vdots &\vdots \\ \mathbf{01}11\cdots 1 & 0\end{array}}  \ar@(dl,dr)[d] \\ 
   C& {\begin{array}{c|r} \mathbf{10}00 \cdots 0 &1 \\ \vdots &\vdots\\ \mathbf{10}11\cdots 1&0 \end{array}} \\
   D& {\begin{array}{c|r} \mathbf{11}00 \cdots 0 &-1 \\ \vdots &\vdots \\ \mathbf{11}11\cdots 1 &0 \end{array}}}
\end{displaymath}

Note that $n$ is either in block $C$ or block $D$.  By induction and the recursive construction of the list, the desired result holds for $n$. \end{proof}

\begin{definition} Given a positive integer $k$, we define the \textbf{chromatic number} of the set $[k] = \{1, 2, \dots, k\}$, denoted by $\chi(k)$, by $\chi(k) = \tau([k])+1.$
\end{definition}

The next lemma refines Lemma~\ref{evil:1}.

\begin{lemma}\label{evil:2} If $k$ is odious and even, then $\chi(k) =2$.  If $k$ is evil and even, then $\chi(k) = 0$.  $\chi(k) = 1$ for every positive odd number $k$.
\end{lemma}

\begin{proof}   Suppose that $k$ has binary expansion $$k = 1\underbrace{00\cdots0}_{z_1\geq0 \ 0\text{'s}}1\underbrace{00\cdots0}_{z_2\geq0 \ 0\text{'s}}10 \cdots 0 1\underbrace{00\cdots0}_{z_{j-1}\geq0 \ 0\text{'s}}1\underbrace{00\cdots0}_{z_{j}\geq0 \ 0\text{'s}},$$  where $k$ has $j$ ones in positions $i_1, i_2, \dots, i_j$ (reading from left-to-right) and the $z_{\ell}$'s give the number of zeros after each one.  For the purposes of this proof, we let $k(i_{\ell})$ denote the number in binary notation derived from $k$ by changing the $i_{\ell}$-th one and any nonzero bit associated to a power of 2 less than it to a zero.

For example, if $k = 10011001$, then $i_1 = 7, i_2 = 4, i_3 = 3, i_4 = 0$, $z_1 = 2, z_2 = 0, z_3 = 2, z_4=0$, and $k(i_3) = k(3) = 10010000$ (note that the one counted by $i_1$ is associated with $2^7$).

Next, for a given positive integer $k$ and an associated $i_{\ell}$, we define $$L(k, i_{\ell}) = k(i_{\ell}) + \left\{\left(\sum_{r=0}^{i_{\ell}-1} c_r 2^r\right)_2 \ : \ c_r = 0 \ \text{or}\  c_r = 1 \right\}$$ if $\ell > 1$ and $$L(k, i_{1}) = \left\{\left(\sum_{r=0}^{i_{1}-1} c_r 2^r\right)_2 \ : \ c_r = 0 \ \text{or}\  c_r = 1 \right\}\setminus \{0\}.$$
For example, if $k = 10011001$, then $$L(k, i_{3}) = 10010000 + \{111,110,101,100,011,010,001,000 \}.$$
We are now ready to proceed with the proof.  We consider two cases.

\noindent{\sc Case 1:} $k$ is odd.  If $k$ is odious, then $\chi(L(k,i_j)) = -1$, since $L(k,i_j)$ consists of exactly one evil number.  Also observe that $\chi(L(k,i_1)) = 1$, since $L(k,i_1) = [2^{i_1 - 1}]\setminus \{0\}.$ However, $\chi(L(k,i_{\ell})) = 0$ for all $1 < \ell < j$ since $\chi([2^q]\cup \{0\}) = 0$ for all $q>0$.  Thus, $\chi(k) = 1 +1 - 1 = 1$, as desired.

If $k$ is evil, then $\chi(L(k,i_j)) = 1$, since $L(k,i_j)$ consists of exactly one odious number.  Also observe that $\chi(L(k,i_1)) = 1$, since $L(k,i_1) = [2^{i_1 - 1}]\setminus \{0\}.$ However, $\chi(L(k,i_{\ell})) = 0$ for all $1 < \ell < j$ since $\chi([2^q]\cup \{0\}) = 0$ for all $q>0$.  Thus, $\chi(k) = -1 +1 + 1 = 1$, as desired.

\bigskip

\noindent{\sc Case 2:} $k$ is even. If $k$ is odious, then $\chi(L(k,i_{\ell})) = 0$ for all $1 < \ell \leq j$ since $\chi([2^q]\cup \{0\}) = 0$ for all $q>0$. As in the case above, we have $\chi(L(k,i_1)) = 1$, since $L(k,i_1) = [2^{i_1 - 1}]\setminus \{0\}.$  Hence, $\chi(k) = 1 + 1 = 2$, as desired.

If $k$ is evil, then $\chi(L(k,i_{\ell})) = 0$ for all $1 < \ell \leq j$ since $\chi([2^q]\cup \{0\}) = 0$ for all $q>0$. Again, we have $\chi(L(k,i_1)) = 1$, since $L(k,i_1) = [2^{i_1 - 1}]\setminus \{0\}.$  Hence, $\chi(k) = -1 + 1 = 0$, as desired.
\end{proof}

  
Recall that the `\rm{mex}' of a nonempty set of nonnegative integers is defined to be the minimum excluded element. For example, if $S = \{0,1,2,7,9,13 \}$, then mex$(S)$ = 3.  Using the mex rule and Lemma~\ref{lem:1}, we characterize the
$\mathcal{P}$-positions of evil-Chromatic Nim in the next theorem.


 \begin{theorem}\label{evil:3}
 The set of $\mathcal{P}$-positions of evil-Chromatic Nim $\{(a_i, b_i) \ | \ i \geq 0\}$ can be computed recursively as follows. The only terminal $\mathcal{P}$-position is $(a_0,b_0) = (0,0)$.  Otherwise, $a_n = {\rm mex}\{a_i, b_i \  | \ i < n\}$  and $b_n$ is the smallest evil number such that $a_n < b_n$.
  \end{theorem}
  
\begin{proof}  The proof of this result follows from Lemma~\ref{lem:1} and Lemma~\ref{evil:1} above. 
\end{proof}
  
  
Given the interesting number theory surrounding evil and odious numbers, we are able to refine the last result quite substantially. We say that a non-negative integer is \emph{vile} if its binary expansion ends in an even number of zeros; otherwise it is called \emph{dopey}.

 \begin{theorem}\label{evil:5}
The set of $\mathcal{P}$-positions of evil-Chromatic Nim $\{(a_i, b_i) \ | \ i \geq 0\}$ are given by $(a_0,b_0) = (0,0)$, and, for $n>0$, by
$$b_n = \left\{ \begin{array}{ll} 2n &\text{if } n \text{ is evil} \\
2n+1 &\text{if } n \text{ is odious}
\end{array} \right.$$ and 

$$a_n = \left\{ \begin{array}{ll} 
b_n-1 &\text{if } n \text{ is evil and dopey}\\
b_n-2 &\text{if } n \text{ is vile} \\
b_n-3 &\text{if } n \text{ is odious and dopey}
\end{array} \right.$$

 \end{theorem}

\begin{proof} The proof will proceed by induction on $n$.  

\bigskip

\noindent{\sc Case 1:} $n$ is \textbf{evil} and \textbf{dopey}.  We claim that $n-1$ must be evil and vile.  If not, then $n-1$ is evil and dopey, odious and dopey, or odious and vile.  In either of the first two cases, $n-1$ has binary expansion $$\underbrace{1** \cdots *1}_{k \ 1\text{'s}} \underbrace{00\cdots 0}_{2m+1\ 0\text{'s}},$$ where $k$ is even if $n-1$ is evil and $k$ is odd if $n-1$ is odious (where $m \geq 0$ and a $\ast$ denotes a 0 or a 1).  But this implies that $n = (n-1) + 1$ is either evil and vile or odious and vile, a contradiction.

Next suppose that $n-1$ is odious and vile.  Then $n-1$ has binary expansion 

$$\begin{array}{lllll} (i)  &\underbrace{1** \cdots *}_{2k+1 \ 1\text{'s}}0\underbrace{11\cdots1}_{2m \ 1\text{'s}}, &{} &(ii) &\underbrace{1** \cdots *}_{2k \ 1\text{'s}}0\underbrace{11\cdots1}_{2m+1 \ 1\text{'s}},\\
{}&{}&{}&{}&{}\\
(iii) &\underbrace{11\cdots1}_{2k+1 \ 1\text{'s}}, &\text{or}  &{(iv)}  & \underbrace{1** \cdots *1}_{2k+1 \ 1\text{'s}}\underbrace{00 \cdots 0}_{2m>0 \ 0\text{'s}}.
\end{array}$$

Since $n$ is dopey, cases $(i)$ and $(iv)$ are not possible, and since $n$ is evil, cases $(ii)$ and $(iii)$ are not possible.

Hence, $n-1$ is evil and vile.  Then $n-1$ has binary expansion 
$$\begin{array}{rrrrrrr} (i)  &\underbrace{1** \cdots *}_{2k+1 \ 1\text{'s}}0\underbrace{11\cdots1}_{2m+1 \ 1\text{'s}}, &(ii) &\underbrace{11\cdots1}_{2k \ 1\text{'s}}, &\text{or}  &{(iii)}  & \underbrace{1** \cdots *1}_{2k \ 1\text{'s}}\underbrace{00 \cdots 0}_{2m>0 \ 0\text{'s}}.
\end{array}$$   Since $n$ is dopey, cases $(ii)$ and $(iii)$ are not possible.   In case $(i)$, $$b_{n-1} = 2(n-1) = (\underbrace{1** \cdots *}_{2k+1 \ 1\text{'s}}0\underbrace{11\cdots1}_{2m+1 \ 1\text{'s}}0)_2,$$ by induction.  Adding the next two consecutive terms after $b_{n-1}$ and using Lemma~\ref{evil:2} we observe that  $$\begin{array}{l|l|crlll}
{} &\chi &{} \\
\hline
b_{n-1} &0 &b_{n-1} &=&2(n-1)\\
a_{i} &1 &{} \\
b_n &0  &b_n &= &b_{n-1} + 2 &= &2n
\end{array}$$  Based on the chromatic numbers shown above, we must have $a_n = a_{i} = b_n-1$. Therefore, if $n$ is evil and dopey, then $b_n = 2n$ and $a_n = b_n - 1$.

\bigskip
\bigskip

\noindent{\sc Case 2:} $n$ is \textbf{odious} and \textbf{dopey}.  We will show that $n-1$ must be odious and vile.  If $n-1$ is evil and dopey, then the binary expansion of $n-1$ looks like $$ \underbrace{1** \cdots *1}_{2k \ 1\text{'s}}\underbrace{00 \cdots 0}_{2m+1 \ 0\text{'s}}.$$  But, if this was true, then $n$ would be vile.  Thus $n-1$ is not evil and dopey.  Next we consider what happens if $n-1$ was evil and vile.  Then $n-1$ would have binary expansion $$\begin{array}{rrrrrrr} (i)  &\underbrace{1** \cdots *}_{2k+1 \ 1\text{'s}}0\underbrace{11\cdots1}_{2m+1 \ 1\text{'s}}, &(ii) &\underbrace{11\cdots1}_{2k \ 1\text{'s}}, &\text{or}  &{(iii)}  & \underbrace{1** \cdots *1}_{2k \ 1\text{'s}}\underbrace{00 \cdots 0}_{2m>0 \ 0\text{'s}}.
\end{array}$$  Case $(i)$ is not possible since $n$ is odious.  Further, cases $(ii)$ and $(iii)$ are not possible because $n$ is dopey.  Now, if $n-1$ is odious and dopey, then $n-1$ would have binary expansion $$ \underbrace{1** \cdots *1}_{2k+1 \ 1\text{'s}}\underbrace{00 \cdots 0}_{2m+1 \ 0\text{'s}}.$$  If this was so, then $n$ would be evil and vile, a contradiction.  By process of elimination, $n-1$ must be odious and vile.  Then $n-1$ has binary expansion 

$$\begin{array}{lllll} 
(i)  &\underbrace{1** \cdots *}_{2k+1 \ 1\text{'s}}0\underbrace{11\cdots1}_{2m \ 1\text{'s}}, &{} &(ii) &\underbrace{1** \cdots *}_{2k \ 1\text{'s}}0\underbrace{11\cdots1}_{2m+1 \ 1\text{'s}},\\
{}&{}&{}&{}&{}\\
(iii) &\underbrace{11\cdots1}_{2k+1 \ 1\text{'s}}, &\text{or}  &{(iv)}  & \underbrace{1** \cdots *1}_{2k+1 \ 1\text{'s}}\underbrace{00 \cdots 0}_{2m>0 \ 0\text{'s}}.
\end{array}$$

Since $n$ is dopey, cases $(i)$ and $(iv)$ are not possible.  In case $(ii)$, $$b_{n-1} = 2(n-1)+1 = (\underbrace{1** \cdots *}_{2k \ 1\text{'s}}0\underbrace{11\cdots1}_{2m+1 \ 1\text{'s}}1)_2,$$ and in case $(iii)$, $$b_{n-1} = 2(n-1)+1 = (\underbrace{11\cdots1}_{2k+1 \ 1\text{'s}}1)_2,$$  both by induction.  Adding the term before $b_{n-1}$ and the two after it,  we see by Lemma~\ref{evil:2} that  $$\begin{array}{l|l|crlll}
{} &\chi &{} \\
\hline
a_{i-1} &2 &{}\\
b_{n-1} &1 &b_{n-1} &=&2(n-1)+1\\
a_{i} &2 &{} \\
b_n &1  &b_n &= &b_{n-1} + 2 &= &2n+1
\end{array}$$  Based on the chromatic numbers shown above, we must have $a_n = a_{i-1} = b_n-3$. Therefore, if $n$ is odious and dopey, then $b_n = 2n+1$ and $a_n = b_n - 3$.

\bigskip
\bigskip

\noindent{\sc Case 3:} $n$ is \textbf{evil} and \textbf{vile}.  We will show that $n-1$ is either odious and dopey or odious and vile.  To this end, if $n-1$ was evil and dopey, then its binary expansion would look like $$\underbrace{1** \cdots *1}_{2k \ 1\text{'s}}\underbrace{00\cdots0}_{2m+1 \ 0\text{'s}}.$$  Since $n$ is evil, this is not possible.  If $n-1$ was evil and vile then its binary expansion would look like 

$$\begin{array}{lllll} 
(i)  &\underbrace{1** \cdots *}_{2k \ 1\text{'s}}0\underbrace{11\cdots1}_{2m \ 1\text{'s}}, &{} &(ii) &\underbrace{1** \cdots *}_{2k+1 \ 1\text{'s}}0\underbrace{11\cdots1}_{2m+1 \ 1\text{'s}},\\
{}&{}&{}&{}&{}\\
(iii) &\underbrace{11\cdots1}_{2k \ 1\text{'s}}, &\text{or}  &{(iv)}  & \underbrace{1** \cdots *1}_{2k \ 1\text{'s}}\underbrace{00 \cdots 0}_{2m>0 \ 0\text{'s}}.
\end{array}$$

Cases $(i), (iii)$, and $(iv)$ are not possible since $n$ is evil and case $(ii)$ is not possible as $n$ is vile.  Thus, $n-1$ is either odious and dopey or odious and vile.  If $n-1$ is odious and dopey, then its binary expansion looks like $$\underbrace{1** \cdots *1}_{2k+1 \ 1\text{'s}}\underbrace{00\cdots0}_{2m+1 \ 0\text{'s}}.$$  Then, by induction, $$b_{n-1} = 2(n-1) + 1 = (\underbrace{1** \cdots *1}_{2k+1 \ 1\text{'s}}\underbrace{00\cdots0}_{2m+1 \ 0\text{'s}}1)_2.$$  Adding the term before $b_{n-1}$ and the one after it,  we see by Lemma~\ref{evil:2} that  $$\begin{array}{l|l|crlll}
{} &\chi &{} \\
\hline
a_i &2 &{}\\
b_{n-1} &1 &b_{n-1} &=&2(n-1)+1\\
b_n &0  &b_n &= &b_{n-1} + 1 &= &2n
\end{array}$$  Based on the chromatic numbers shown above, we must have $a_n = a_i = b_n-2$.    Now if $n-1$ is odious and vile, then its  binary expansion looks like 

$$\begin{array}{lllll} 
(i)  &\underbrace{1** \cdots *}_{2k+1 \ 1\text{'s}}0\underbrace{11\cdots1}_{2m \ 1\text{'s}},&{}  &(ii) &\underbrace{1** \cdots *}_{2k \ 1\text{'s}}0\underbrace{11\cdots1}_{2m+1 \ 1\text{'s}},\\
{}&{}&{}&{}&{}\\
(iii) &\underbrace{11\cdots1}_{2k+1 \ 1\text{'s}}, &\text{or}  &{(iv)}  & \underbrace{1** \cdots *1}_{2k+1 \ 1\text{'s}}\underbrace{00 \cdots 0}_{2m >0 \ 0\text{'s}}.
\end{array}$$

Because $n$ is vile, cases $(ii)$ and $(iii)$ are not possible.  By induction $$b_{n-1} = 2(n-1)+1 = (\underbrace{1** \cdots *}_{2k+1 \ 1\text{'s}}0\underbrace{11\cdots1}_{2m \ 1\text{'s}}1)_2$$ in case $(i)$ and $$b_{n-1} = 2(n-1)+1 = (\underbrace{1** \cdots *1}_{2k+1 \ 1\text{'s}}\underbrace{00 \cdots 0}_{2m >0 \ 0\text{'s}}1)_2$$ in case $(iv)$.  In either case we again have $$\begin{array}{l|l|crlll}
{} &\chi &{} \\
\hline
a_i &2 &{}\\
b_{n-1} &1 &b_{n-1} &=&2(n-1)+1\\
b_n &0  &b_n &= &b_{n-1} + 1 &= &2n
\end{array}$$  Thus, if $n$ is evil and vile, then $b_n = 2n$ and $a_n = b_n - 2$.

\bigskip
\bigskip

\noindent{\sc Case 4:} $n$ is \textbf{odious} and \textbf{vile}.  We will show that $n-1$ is either evil and dopey or evil and vile.  If $n-1$ was odious and dopey, then its binary expansion would look like $$\underbrace{1** \cdots *1}_{2k+1 \ 1\text{'s}}\underbrace{00\cdots0}_{2m+1 \ 0\text{'s}}.$$  Since $n$ is odious, this is not possible.  If $n-1$ was odious and vile, then it binary expansion would look like $$\begin{array}{lllll} (i)  &\underbrace{1** \cdots *}_{2k+1 \ 1\text{'s}}0\underbrace{11\cdots1}_{2m \ 1\text{'s}}, &{} &(ii) &\underbrace{1** \cdots *}_{2k \ 1\text{'s}}0\underbrace{11\cdots1}_{2m+1 \ 1\text{'s}},\\

{}&{}&{}&{}&{}\\

(iii) &\underbrace{11\cdots1}_{2k+1 \ 1\text{'s}}, &\text{or} &{(iv)}  &\underbrace{1** \cdots *1}_{2k+1 \ 1\text{'s}}\underbrace{00 \cdots 0}_{2m >0 \ 0\text{'s}}.\end{array}$$

Because $n$ is odious, cases $(i)$ and $(iv)$ are not possible and since $n$ is vile, cases $(ii)$ and $(iii)$ are not possible.  Thus, $n-1$ is either evil and dopey or evil and vile.  If $n-1$ is evil and dopey, then it has binary expansion $$\underbrace{1**\cdots *1}_{2k \ 1\text{'s}}\underbrace{00\cdots0}_{2m+1 \ 0\text{'s}}.$$  Thus, by induction, $$b_{n-1} = 2(n-1) =  (\underbrace{1**\cdots *1}_{2k \ 1\text{'s}}\underbrace{00\cdots0}_{2m+1 \ 0\text{'s}}0)_2.$$ Adding the three consecutive terms after $b_{n-1}$ and using Lemma~\ref{evil:2} we have
$$\begin{array}{l|l|crlll}
{} &\chi &{} \\
\hline
b_{n-1} &0 &b_{n-1} &=&2(n-1)\\
a_{i-1} &1 \\
a_{i} &2 &{} \\
b_n &1  &b_n &= &b_{n-1} + 3 &= &2n + 1
\end{array}$$  Based on the chromatic numbers shown, we must have $a_n = a_{i-1} = b_n-2$.  If $n-1$ is evil and vile, then it has binary expansion 

$$\begin{array}{lllll} 
(i)  &\underbrace{1** \cdots *}_{2k \ 1\text{'s}}0\underbrace{11\cdots1}_{2m \ 1\text{'s}},&{} &(ii) &\underbrace{1** \cdots *}_{2k+1 \ 1\text{'s}}0\underbrace{11\cdots1}_{2m+1 \ 1\text{'s}},\\
{}&{}&{}&{}&{}\\
(iii) &\underbrace{11\cdots1}_{2k \ 1\text{'s}}, &\text{or}  &{(iv)}  & \underbrace{1** \cdots *1}_{2k \ 1\text{'s}}\underbrace{00 \cdots 0}_{2m>0 \ 0\text{'s}}.
\end{array}$$

Since $n$ is odious, $(ii)$ is not possible.  By induction $$b_{n-1} = 2(n-1) = (\underbrace{1** \cdots *}_{2k \ 1\text{'s}}0\underbrace{11\cdots1}_{2m \ 1\text{'s}}0)_2$$ in case $(i)$, $$b_{n-1} = 2(n-1) = (\underbrace{11\cdots1}_{2k \ 1\text{'s}}0)_2$$ in case $(iii)$, and 
$$b_{n-1} = 2(n-1) = (\underbrace{1** \cdots *1}_{2k \ 1\text{'s}}\underbrace{00 \cdots 0}_{2m >0 \ 0\text{'s}}0)_2$$ in case $(iv)$.  No matter what the case, we again have $$\begin{array}{l|l|crlll}
{} &\chi &{} \\
\hline
b_{n-1} &0 &b_{n-1} &=&2(n-1)\\
a_{i-1} &1 \\
a_{i} &2 &{} \\
b_n &1  &b_n &= &b_{n-1} + 3 &= &2n + 1
\end{array}$$   Thus $a_n = a_{i-1} = b_n-2$.  Hence, if $n$ is odious and vile, then $b_n = 2n+1$ and $a_n = b_n-2$.    \end{proof}

\begin{example}
Find the $17509^{17509}$th $\mathcal{P}$-position of evil-Chromatic Nim ($17509$ is the $2015$th prime).  
\end{example}

\noindent With the help of the computer algebra system \emph{Mathematica} we know that $q=17509^{17509}$ is evil and vile.  Hence, Theorem~\ref{evil:5} tells us that $$(a_q,b_q) = (2q-2, 2q).$$ \hfill$\Diamond$


\begin{thebibliography}{11}
\bibitem[BCG]{BCG} E. R. Berlekamp, J. H. Conway, R. K. Guy,  \emph{Winning ways}, {\bf 1-2} Academic Press, London (1982). Second edition, {\bf 1-4}. A. K. Peters, Wellesley/MA (2001/03/03/04).
\bibitem[Be]{Be} S. Beatty, Problem 3173, \emph{Amer. Math. Monthly}, {\bf 33} (1926) 159.
\bibitem[Bo]{Bo} C. L. Bouton, Nim, A Game with a Complete Mathematical Theory
\emph{The Annals of Mathematics}, 2nd Ser., Vol. 3, No. 1/4. (1901 - 1902), 35--39.
\bibitem[DR]{DR} E. Duch\^{e}ne and M. Rigo, Invariant Games,
\emph{Theoret. Comp. Sci.}, Vol. 411, 34-36 (2010),  3169--3180 
\bibitem[F82]{F82} A.\ S.\ Fraenkel, How to beat your Wythoff games' opponent on three fronts, {\it Amer. Math. Monthly\/} {\bf 89} (1982) 353--361.
\bibitem[F]{F} A.\ S.\ Fraenkel, The Rat game and the Mouse game, to appear in \emph{Games of no Chance 2008}.
\bibitem[FL]{FL} A.\ S.\ Fraenkel, U.\ Larsson, Take-away games on Beatty's theorem and the notion of $k$-invariance, GONC5.
\bibitem[G]{G} S. W.\ Golomb, A mathematical investigation of games of "take-away''. \emph{J. Combinatorial Theory} {\bf 1} (1966) 443--458.
\bibitem[L1]{L1} U.\ Larsson Restrictions of $m$-Wythoff Nim and $p$-complementary Beatty sequences, to appear in \emph{Games of no Chance 4}.
\bibitem[L2]{L2} U.\ Larsson, 2-pile Nim with a Restricted Number of Move-size Imitations, \emph{Integers} {\bf 9} (2009), Paper G4, 671--690.
\bibitem[L3]{L3} U.\ Larsson, Blocking Wythoff Nim, \emph{The Electronic Journal of Combinatorics}, P120 of Volume 18(1) (2011).
\bibitem[L2012]{L2012} U. Larsson, The $\star$-operator and invariant subtraction games, \emph{Theoret. Comput. Sci.}, Vol. 422, (2012) 52--58.
\bibitem[L2014]{L2014} U. Larsson, Wythoff Nim extensions and splitting sequences, \emph{Journal of Integer Sequences}, Vol. 17 (2014) Article 14.5.7
\bibitem[LHF]{LHF} U.\ Larsson, P.\ Hegarty, A.\ S.\ Fraenkel, Invariant and dual subtraction games resolving the Duch\^ene-Rigo Conjecture, \emph{Theoret. Comp. Sci.} Vol. 412, 8-10 (2011) 729--735.
\bibitem[LW]{LW} U.\ Larsson, M.\ Weimerskirch, Impartial games, whose rule sets produce given continued fractions, preprint.  
\bibitem[R]{R} J. W. Rayleigh. The Theory of Sound, \emph{Macmillan, London}, (1894) 122--123.
\bibitem[S]{S} A. J. Schwenk, ``Take-Away Games", \emph{Fibonacci Quart.} {\bf 8} (1970), 225--234.
\bibitem[Z]{Z} Michael Zieve, Take-Away Games, \emph{Games of No Chance, MSRI Publications}, {\bf 29}, (1996) 351--361
\bibitem[W]{W} W.A. Wythoff, A modification of the game of Nim, \emph{Nieuw Arch. Wisk.} {\bf 7} (1907) 199--202.
\end{thebibliography}
\end{document}